\documentclass[11pt,a4paper,twoside]{article}

\usepackage{amssymb,amscd,amsfonts,amsbsy}
\usepackage{enumerate}
\usepackage{amsmath,amsthm}
\usepackage{mathrsfs}
\usepackage{epsf,epsfig}
\usepackage{pdfsync}
\usepackage[all]{xy}

\numberwithin{equation}{section}

\newcommand{\fra}{\mathfrak{a}}

\newtheorem{proposition}{Proposition}[section]
\newtheorem{theorem}[proposition]{Theorem}
\newtheorem{lemma}[proposition]{Lemma}
\newtheorem{corollary}[proposition]{Corollary}

\theoremstyle{definition}

\newtheorem{remark}[proposition]{Remark}

\title{The incompressible Navier-Stokes system\\ with time-dependent  Robin-type  boundary conditions}
\author{Sylvie Monniaux\,\thanks{Aix-Marseille Universit\'e, CNRS, Centrale Marseille, 
I2M, UMR 7373 -- 13453 Marseille, France}\,  \footnotemark[3] 
\and El Maati Ouhabaz\,\thanks{Universit\'e de Bordeaux,
 Institut de Math\'ematiques de Bordeaux IMB, CNRS UMR 5251 -- 33405 Bordeaux, France}\,
\thanks{The research of both authors was partially supported by the ANR project HAB, 
ANR-12-BS01-0013-02 and ANR-12-BS01-0013-03}}

\date{}

\begin{document}

\maketitle

\begin{abstract}
We show that the incompressible 3D Navier-Stokes system in a ${\mathscr{C}}^{1,1}$ 
bounded domain or a bounded convex domain $\Omega$ with a non penetration condition 
$\nu\cdot u=0$ at the boundary $\partial\Omega$ together with a time-dependent Robin 
boundary condition of the type $\nu\times{\rm curl}\,u=\beta(t) u$ on $\partial\Omega$ admits a 
solution with enough regularity provided the initial condition is small enough in an appropriate 
functional space. 
\end{abstract}

\section{Introduction} 
\label{sec:intro}

We consider the following incompressible Navier-Stokes system in a (sufficiently smooth) bounded
domain $\Omega\subset{\mathbb{R}}^3$ on a time interval $[0,\tau]$ 
\begin{equation}  
\label{NS} \tag{NS}
\left\{
\begin{array}{rclcl}
\partial_t u-\Delta u + \nabla p +(u\cdot\nabla)u &=& 0 &\mbox{in}&[0,\tau]\times\Omega
\\[4pt]
{\rm div}\,u &=& 0 &\mbox{in}&[0,\tau]\times\Omega 
\end{array}\right.
\end{equation}
where ${\mathbb{S}}(u,p):= \frac{1}{2}\bigl(\nabla u+(\nabla u)^\top\bigr)-p\,{\rm Id}$ is the 
Cauchy stress tensor applied to $(u,p)$ supplemented with the conditions on the boundary
$\partial\Omega$ ($\nu$ denotes the outer unit normal):
\begin{equation}
\label{NBC}\tag{Nbc}
\left\{
\begin{array}{rclcl}
\nu \cdot u &=& 0 &\mbox{on} &[0,\tau]\times\partial\Omega 
\\[4pt]
\bigl[{\mathbb{S}}(u,p)\nu\bigr]_{\rm tan} + B u &=&0&\mbox{on} &[0,\tau]\times\partial\Omega
\end{array}
\right.
\end{equation}
and the initial condition 
\begin{equation}
\label{IC}\tag{IC}
u(0)=u_0\quad\mbox{in }\Omega.
\end{equation}
As usual $[w]_{\rm tan}$ denotes the tangential part of $w$, that is $[w]_{\rm tan}=w-(\nu\cdot w)\nu$.
The conditions (Nbc) are referred to in the literature as Navier's boundary conditions and
were introduced by Navier in his lecture at the Acad\'emie royale des Sciences in 1822 \cite{Na1822}. 
They describe the fact that the fluid cannot escape from the domain $\Omega$ ($\nu\cdot u=0$) and 
that the fluid slips with a friction described by a matrix $B$ on $\partial\Omega$
($\bigl[{\mathbb{S}}(u,p)\nu\bigr]_{\rm tan} + B u=0$). Such conditions have been recently derived from
homogenization of rough boundaries, see e.g. \cite{JM01}, \cite{BGV08}, \cite{GVM10}, \cite{BFN10}.

First we transform the system (NS) with boundary conditions (Nbc) and initial condition (IC) 
into the following ``Robin-Navier-Stokes'' problem
\begin{equation}  
\label{RNS} \tag{RNS}
\left\{
\begin{array}{rclcl}
\partial_t u-\Delta u + \nabla \pi -u\times{\rm curl}\,u &=& 0 &\mbox{in}&[0,\tau]\times\Omega
\\[4pt]
{\rm div}\,u &=& 0 &\mbox{in}&[0,\tau]\times\Omega 
\\[4pt]
\nu \cdot u \ =\ 0,\quad \nu\times{\rm curl}\,u &=&\beta u&\mbox{on} &[0,\tau]\times\partial\Omega
\\[4pt]
u(0)&=&u_0&\mbox{in}&\Omega.
\end{array}\right.
\end{equation}
This is based on the identities $(u\cdot\nabla)u=-u\times{\rm curl}\,u+\frac{1}{2}\nabla|u|^2$
and $\bigl[{\mathbb{S}}(u,p) \nu\bigr]_{\rm tan} = - \nu \times {\rm curl}\,u +2{\mathcal{W}}u$ 
on the boundary $\partial \Omega$ (see, e.g., \cite[Section~2]{MM09b}), so that
$\beta=2{\mathcal{W}}+B$, and $\pi=p+\frac{1}{2}|u|^2$.
Here ${\mathcal{W}}$ is the Weingarten map (for properties of ${\mathcal{W}}$, see, e.g., 
\cite[Section~6]{MM09b}; in particular, ${\mathcal{W}}u = 0$ on flat parts of the boundary).
We prove, in the Hilbert space setting, existence and uniqueness of solutions of (RNS)
for time-dependent and boundary-dependent symmetric positive matrices 
$\beta:[0, \tau]\times\partial\Omega\to{\mathscr{M}}_3({\mathbb{R}})$ uniformly bounded in 
$x\in\partial\Omega$ and piecewise H\"older-continuous in $t\in[0,\tau]$. For precise 
hypotheses on $\beta$, we refer to Section~\ref{secRS} and Section~\ref{sec:main} below.
Note that the condition $\beta\ge0$ implies the geometric condition on the friction 
(symmetric) matrix $B$: $B\ge-2{\mathcal{W}}$. In particular, if $\Omega$ is convex,
${\mathcal{W}}\ge 0$, so that we can treat any nonnegative friction matrix $B$.
The main result  is the following

\begin{theorem}
\label{thm-intro}
Let $\Omega\subset{\mathbb{R}}^3$ be a bounded ${\mathscr{C}}^{1,1}$ or convex domain
and let $\tau>0$. There exists $\epsilon>0$ such that for all
initial condition $u_0 \in L^2(\Omega,{\mathbb{R}}^3)$ with ${\rm div}\,u_0=0$
in $\Omega$, $\nu\cdot u_0=0$ on $\partial\Omega$ and 
${\rm curl}\,u_0\in L^2(\Omega,{\mathbb{R}}^3)$, 
$\|u_0\|_2+\|{\rm curl}\,u_0\|_2 \le \epsilon$, there exists a unique  $(u, \pi)$ satisfying
(RNS) for a.e. $(t,x) \in [0, \tau] \times \Omega$. In addition, 
$u \in H^1(0,\tau,L^2(\Omega,{\mathbb R}^3))$, $\Delta u \in L^2(0,\tau,L^2(\Omega, {\mathbb R}^3))$, 
$\pi \in L^2(0,\tau,H^1(\Omega))$ and there exists 
a constant $C$ independent of $u$ and $\pi$ such that 
$$
\|u\|_{H^1(0,\tau,L^2(\Omega,{\mathbb R}^3))} + \|-\Delta u\|_{L^2(0,\tau,L^2(\Omega,{\mathbb{R}}^3))}
+\|\nabla\pi\|_{L^2(0,\tau,L^2(\Omega,{\mathbb{R}}^3))} \le C \epsilon.
$$
\end{theorem}

In the case where $\beta(t,x)=0$ for all $(t,x)\in[0,\tau]\times\partial\Omega$, the system (RNS)
has been studied in \cite{MM09b}, in the case of Lipschitz domains for initial conditions in $L^3$. 
For Dirichlet boundary conditions $u=0$ on $\partial\Omega$,
which correspond to $\beta=\infty$, we refer to the classical results by Fujita and Kato \cite{FK64}
(see also \cite{Mo06}, \cite{MM08} for the case of less regular domains).

The method to prove Theorem~\ref{thm-intro} relies on the study of operators defined by forms
and recent results on maximal regularity  for non-autonomous linear evolution equations. This latter 
property is the key ingredient to treat the non linearity by appealing to classical fixed point arguments.

The paper is organized as follows. Section~\ref{sec:def} is devoted to analytical tools necessary 
for our approach of the problem. In Section~\ref{secRS}, we define the (time dependent) 
Robin Stokes operator. We use recent results on maximal regularity in Section~\ref{sec:main} in 
order to obtain regularity properties of the solution of the linearized (RNS) system. The proof of 
Theorem~\ref{thm-intro} is given in Section~\ref{secNS}.


\section{Background material}
\label{sec:def}

Throughout this section,  $\Omega\subset{\mathbb{R}}^3$ will be a bounded domain which is 
either convex or ${\mathscr{C}}^{1,1}$. 
We denote by $\partial\Omega$ its boundary. It is endowed with the surface measure 
${\rm d}\sigma$. 
It is a classical fact (see, e.g., \cite[Th\'eor\`eme~8.3]{LiMa68} for smooth domains
and \cite[Ch.\,2, Th\'eor\`eme~5.5]{Necasf} or \cite[Ch.\,2, Theorem~5.5]{Necas} for Lipschitz domains)
$$
{\rm Tr}_{|_{\partial\Omega}} : H^1(\Omega) \to  H^{1/2}(\partial\Omega) \hookrightarrow 
L^2(\partial\Omega, {\rm d}\sigma),
$$
the latter embedding being compact. 
\begin{itemize}
\item[$(i)$]
For $u\in L^2(\Omega,{\mathbb{R}}^3)$ such that ${\rm div}\,u \in L^2(\Omega)$, the 
normal component $\nu\cdot u$ of $u$ on $\partial\Omega$ is defined in a weak sense 
in the negative Sobolev space $H^{-1/2}(\partial\Omega)$ by
\begin{equation}
\label{nudotu}
_{H^{-1/2}(\partial\Omega)}\langle \nu\cdot u,\varphi\rangle_{H^{1/2}(\partial\Omega)} 
= \langle {\rm div}\,u,\phi\rangle_\Omega + \langle u,\nabla \phi\rangle_\Omega,
\end{equation}
for all $\varphi \in H^{1/2}(\partial\Omega)$, where $\phi$ belongs to the Sobolev space
$H^1(\Omega)$ with ${\rm Tr}_{|_{\partial\Omega}} \phi=\varphi$.
Here, $\langle\cdot,\cdot\rangle_\Omega$
denotes either the scalar or the vector-valued scalar product in $L^2$ 
defined over $\Omega$. The notation $_{V'}\langle \cdot,\cdot \rangle_V$ means the duality 
between $V'$ and $V$. 
\item[$(ii)$]
For $u\in L^2(\Omega,{\mathbb{R}}^3)$ such that  ${\rm curl}\,u\in L^2(\Omega,{\mathbb{R}}^3)$, 
the tangential component $\nu\times u$ of $u$ on $\partial\Omega$ is defined in a weak sense 
in $H^{-1/2}(\partial\Omega,{\mathbb{R}}^3)$ by
\begin{equation}
\label{nutimesu}
_{H^{-1/2}(\partial\Omega,{\mathbb{R}}^3)}\langle \nu\times u,
\varphi\rangle_{H^{1/2}(\partial\Omega,{\mathbb{R}}^3)} 
= \langle {\rm curl}\,u,\phi\rangle_\Omega - \langle u, {\rm curl}\,\phi\rangle_\Omega,
\end{equation}
for all $\varphi \in H^{1/2}(\partial\Omega,{\mathbb{R}}^3)$ where 
$\phi \in H^1(\Omega,{\mathbb{R}}^3)$ with ${\rm Tr}_{|_{\partial\Omega}} \phi=\varphi$.
As before, $\langle\cdot,\cdot\rangle_\Omega$
denotes the vector-valued scalar product in $L^2$ 
defined over $\Omega$.
\end{itemize}

The following result, valid for Lipschitz domains, can be found in \cite{Cos90} (see also \cite{Mo14}). 

\begin{proposition}
\label{sobolev}
There exists a constant $C>0$ such that for all $u\in L^2(\Omega,{\mathbb{R}}^3)$ satisfying
${\rm div}\,u\in L^2(\Omega,{\mathbb{R}})$, ${\rm curl}\,u\in L^2(\Omega,{\mathbb{R}}^3)$
and either  $\nu\cdot u\in L^2(\partial\Omega)$ or 
$\nu\times u\in L^2(\partial\Omega,{\mathbb{R}}^3)$  we have 
${\rm Tr}_{|_{\partial\Omega}} u \in L^2(\partial \Omega,{\mathbb{R}}^3)$ with the estimate 
$$
\|{\rm Tr}_{|_{\partial\Omega}} u \|_{L^2(\partial\Omega,{\mathbb{R}}^3)} 
\le C\,\Bigl(\|u\|_2 + \|{\rm div}\,u\|_2 + \|{\rm curl}\,u\|_2
+\min\bigl\{\|\nu\cdot u\|_{L^2(\partial\Omega)}, \|\nu\times u\|_{L^2(\partial\Omega,{\mathbb{R}}^3)}\bigr\}
\Bigr).
$$
Moreover, $u\in H^{1/2}(\Omega,{\mathbb{R}}^3)$ and
\begin{equation}
\label{sobolev1}
\|u\|_{H^{1/2}(\Omega,{\mathbb{R}}^3)}\le 
C\,\Bigl(\|u\|_2 + \|{\rm div}\,u\|_2 + \|{\rm curl}\,u\|_2
+\min\bigl\{\|\nu\cdot u\|_{L^2(\partial\Omega)}, \|\nu\times u\|_{L^2(\partial\Omega,{\mathbb{R}}^3)}\bigr\}
\Bigr).
\end{equation}
\end{proposition}

Moving on, let $W_T$ and $W_N$ be the spaces defined by
\begin{equation}
\label{defWT}
W_T=\bigl\{u\in L^2(\Omega,{\mathbb{R}}^3); {\rm div}\,u\in L^2(\Omega), 
{\rm curl}\,u\in L^2(\Omega,{\mathbb{R}}^3) \mbox{ and } \nu\cdot u=0 \mbox{ on }\partial\Omega\bigr\}
\end{equation}
and 
\begin{equation}
\label{defWN}
W_N=\bigl\{u\in L^2(\Omega,{\mathbb{R}}^3); {\rm div}\,u\in L^2(\Omega), 
{\rm curl}\,u\in L^2(\Omega,{\mathbb{R}}^3) \mbox{ and } \nu\times u=0 \mbox{ on }\partial\Omega\bigr\}
\end{equation}
both endowed with the norm
\begin{equation}
\label{normW}
\|u\|_W=\|u\|_{L^2(\Omega,{\mathbb{R}}^3)}+\|{\rm div}\,u\|_{L^2(\Omega)}
+\|{\rm curl}\,u\|_{L^2(\Omega,{\mathbb{R}}^3)}, \quad u\in W.
\end{equation}
It is easy to see that $W_{T,N}$ are Hilbert spaces. Note also that since $\Omega$ is either convex or 
${\mathscr{C}}^{1,1}$, the spaces $W_{T,N}$ are contained in $H^1(\Omega,{\mathbb{R}}^3)$ 
(with continuous embedding). See \cite[Theorem~2.9, Theorem~2.12 and Theorem~2.17]{ABDG98}.
Thus, there exists a constant $C > 0$ such that for all $u \in W_{T,N}$
\begin{equation}
\label{WdansH1}
\|u\|_{H^1(\Omega)} \le C (\|u\|_{L^2(\Omega,{\mathbb{R}}^3)}+\|{\rm div}\,u\|_{L^2(\Omega)}
+\|{\rm curl}\,u\|_{L^2(\Omega,{\mathbb{R}}^3)}).
\end{equation}
In particular, the trace operator  
$$
{\rm Tr}_{|_{\partial\Omega}} : W_{T,N} \to H^{1/2}(\partial\Omega,{\mathbb{R}}^3)
$$
is continuous.

\vspace{.2cm}
Next, we define the Hodge Laplacians with absolute and relative boundary conditions. 
Although these operators do not appear explicitly in our main results they will be useful 
for the proof of the description of the domain of Stokes operator with time dependent Robin 
boundary condition. 
 
We define on $L^2(\Omega,{\mathbb{R}}^3)$ the two bilinear symmetric forms
\begin{equation}
\label{b0}
b_0(u,v)=\langle {\rm div}\,u,{\rm div}\,v\rangle_\Omega+
\langle {\rm curl}\,u,{\rm curl}\,v\rangle_\Omega,\quad u,v\in W_T
\end{equation}
and
\begin{equation}
\label{b1}
b_1(u,v)=\langle {\rm div}\,u,{\rm div}\,v\rangle_\Omega+
\langle {\rm curl}\,u,{\rm curl}\,v\rangle_\Omega,\quad u,v\in W_N.
\end{equation}
Both forms $b_0$ and $b_1$ are closed. Therefore, there exist two operators $B_{0,0}: W_T \to W_T'$ 
associated with $b_0$ ($B_{0,0}u=-\Delta u$) and $B_{1,0}: W_N \to W_N'$ 
($B_{1,0}u=-\Delta u$) associated with $b_1$ in the sense that 
$$
b_0(u,v) =\  _{W_T'}\langle B_{0,0}u, v \rangle_{W_T}, \quad u, v \in W_T
$$
and
$$
b_1(u,v) =\  _{W_N'}\langle B_{1,0}u, v \rangle_{W_N}, \quad u, v \in W_N.
$$
The part $B_0$ of $B_{0,0}$ on $L^2(\Omega,{\mathbb{R}}^3)$, i.e.,
\begin{equation}
\label{def:B0}
D(B_0) := \{ u \in W_T, \exists\,v \in L^2(\Omega,{\mathbb{R}}^3): b_0(u,\phi) = 
\langle v, \phi \rangle_\Omega \ \forall\,\phi \in W_T \},\quad  B_0 u := v,
\end{equation}
and the part $B_1$ of $B_{1,0}$ on $L^2(\Omega,{\mathbb{R}}^3)$, i.e.,
\begin{equation}
\label{def:B1}
D(B_1) := \{ u \in W_N, \exists\,v \in L^2(\Omega,{\mathbb{R}}^3): b_1(u,\phi) = 
\langle v, \phi \rangle_\Omega \ \forall\,\phi \in W_N \},\quad  B_1 u := v,
\end{equation}
are self-adjoint operators on $L^2(\Omega,{\mathbb{R}}^3)$. 

\begin{proposition}
\label{curlB0=B1curl}
The domains of $B_0$ and $B_1$ have the following description
\begin{align}
\label{D(B0)}
D(B_0)=&\bigl\{u\in L^2(\Omega,{\mathbb{R}}^3); {\rm div}\,u\in H^1(\Omega), 
{\rm curl}\,u\in L^2(\Omega,{\mathbb{R}}^3), {\rm curl}\,{\rm curl}\,u\in L^2(\Omega,{\mathbb{R}}^3)
\nonumber\\[4pt]
&\mbox{ and }\nu\cdot u=0, \nu\times{\rm curl}\,u=0\mbox{ on }\partial\Omega\bigr\}
\end{align}
and
\begin{align}
\label{D(B1)}
D(B_1)=&\bigl\{u\in L^2(\Omega,{\mathbb{R}}^3); {\rm div}\,u\in H^1(\Omega), 
{\rm curl}\,u\in L^2(\Omega,{\mathbb{R}}^3), {\rm curl}\,{\rm curl}\,u\in L^2(\Omega,{\mathbb{R}}^3)
\nonumber\\[4pt]
&\mbox{ and }\nu\times u=0, {\rm div}\,u=0\mbox{ on }\partial\Omega\bigr\}.
\end{align}
Moreover, for $u\in L^2(\Omega,{\mathbb{R}}^3)$ such that 
${\rm curl}\,u\in L^2(\Omega,{\mathbb{R}}^3)$,
the following commutator property occurs for all $\varepsilon>0$
\begin{equation}
\label{curlB}
{\rm curl}\,(1+\varepsilon B_0)^{-1}u=(1+\varepsilon B_1)^{-1}{\rm curl}\,u.
\end{equation}
\end{proposition}

\begin{proof}
The description of the domain of $B_0$ can be found in \cite[(3.17) \& (3.18)]{MM09a}. 
We can describe the domain of $B_1$ in the same way (see also 
\cite[Theorem~7.1 \& Theorem~7.3]{Mi04}). To prove
\eqref{curlB}, let $u\in L^2(\Omega,{\mathbb{R}}^3)$ such that 
${\rm curl}\,u\in L^2(\Omega,{\mathbb{R}}^3)$. Let $u_\varepsilon=(1+\varepsilon B_0)^{-1}u$
and $w_\varepsilon=(1+\varepsilon B_1)^{-1}{\rm curl}\,u$. 

\medskip

\noindent
{\tt Step~1}: {\em We claim that ${\rm curl}\,u_\varepsilon\in D(B_1)$.}

\noindent
By \eqref{D(B0)} we have ${\rm curl}\,u_\varepsilon\in L^2(\Omega,{\mathbb{R}}^3)$, 
${\rm curl}\,{\rm curl}\,u_\varepsilon\in L^2(\Omega,{\mathbb{R}}^3)$, 
${\rm div}\,({\rm curl}\,u_\varepsilon)=0\in H^1(\Omega)$, $\nu\times{\rm curl}\,u_\varepsilon=0$
on $\partial\Omega$ and ${\rm div}\,({\rm curl}\,u_\varepsilon)=0$ on $\partial\Omega$. To prove
that ${\rm curl}\,u_\varepsilon \in D(B_1)$, it remains to show, thanks to \eqref{D(B1)}, that 
${\rm curl}\,{\rm curl}\,({\rm curl}\,u_\varepsilon)\in L^2(\Omega,{\mathbb{R}}^3)$. This is due to
the fact that
$$
{\rm curl}\,{\rm curl}\,({\rm curl}\,u_\varepsilon)={\rm curl}\,(-\Delta u_\varepsilon)\quad
\mbox{in } H^{-1}(\Omega,{\mathbb{R}}^3).
$$
Since 
$$
-\Delta u_{\varepsilon}=B_0(1+\varepsilon B_0)^{-1}u=\frac{1}{\varepsilon}\,\bigl(u-u_\varepsilon\bigr)
$$
and ${\rm curl}\,u_\varepsilon, {\rm curl}\,u\in L^2(\Omega,{\mathbb{R}}^3)$, the claim follows.

\medskip

\noindent
{\tt Step~2}: {\em We claim now that ${\rm curl}\,u_\varepsilon=w_\varepsilon$.}

\noindent
By Step~1, we know that ${\rm curl}\,u_\varepsilon\in D(B_1)$. Moreover, we have in the sense 
of distributions
$$
(1+\varepsilon B_1)({\rm curl}\,u_\varepsilon)=
{\rm curl}\,u_\varepsilon-\varepsilon\Delta {\rm curl}\,u_\varepsilon
={\rm curl}\,\Bigl(u_\varepsilon-\varepsilon\Delta u_\varepsilon\Bigr)={\rm curl}\,u
$$
since $u_\varepsilon-\varepsilon\Delta u_\varepsilon=(1+\varepsilon B_0)(1+\varepsilon B_0)^{-1}u=u$.
Therefore, 
$$
{\rm curl}\,u_\varepsilon =(1+\varepsilon B_1)^{-1}{\rm curl}\,u=w_\varepsilon
$$
which proves the claim.
\end{proof}

The following lemma is inspired by \cite[Proof of Proposition~2.4 (iii)]{Mi04}.

\begin{lemma}
\label{extnutimesu}
\begin{enumerate}
\item
\label{1}
Let $g\in L^2(\partial\Omega,{\mathbb{R}}^3)$. 
Then there exists $w\in L^2(\Omega,{\mathbb{R}}^3)$ with 
${\rm curl}\,w\in L^2(\Omega,{\mathbb{R}}^3)$ such that for all $\phi\in W_T$
\begin{equation}
\label{extnutimesu1}
\langle g,\phi\rangle_{\partial\Omega}=
\langle {\rm curl}\,w,\phi\rangle_\Omega-\langle w,{\rm curl}\,\phi\rangle_\Omega.
\end{equation}
Moreover, there exists $C>0$ such that 
\begin{equation}
\label{estw}
\|w\|_{L^2(\Omega,{\mathbb{R}}^3)}+\|{\rm curl}\,w\|_{L^2(\Omega,{\mathbb{R}}^3)}
\le C\|g\|_{L^2(\partial\Omega,{\mathbb{R}}^3)}.
\end{equation}
\item
\label{2}
If  in addition  $g\in L^2_{\rm tan}(\partial\Omega,{\mathbb{R}}^3)$
(which means that $g \in L^2(\partial\Omega,{\mathbb{R}}^3)$ and
$\nu\cdot g=0$ on $\partial\Omega$), then there exists $w\in L^2(\Omega,{\mathbb{R}}^3)$
such that ${\rm curl}\,w\in L^2(\Omega,{\mathbb{R}}^3)$ and
\eqref{extnutimesu1} holds for all $\phi\in H^1(\Omega)$. And in that case
$g=\nu\times w$ in $H^{-1/2}(\partial\Omega,{\mathbb{R}}^3)$.
\end{enumerate}
\end{lemma}

\begin{proof}
\ref{1}.
We define the space $X:= \{(\phi,{\rm curl}\,\phi); \phi \in W_T \}$. It is a closed subspace of 
$L^2(\Omega,{\mathbb{R}}^3)\times L^2(\Omega,{\mathbb{R}}^3)$.  
By classical trace theorems 
(\cite[Th\'eor\`eme~8.3]{LiMa68}, \cite[Ch.\,2, Th\'eor\`eme~5.5]{Necasf} or
\cite[Ch.\,2, Theorem~5.5 with $k=1$ and $p=2$]{Necas}), we have
that $\nu\times \phi\in L^2(\partial\Omega,{\mathbb{R}}^3)$ for all 
$\phi\in W_T \subset H^1(\Omega,{\mathbb{R}}^3)$. 
Since $g\in L^2(\partial\Omega,{\mathbb{R}}^3)$, it is immediate that
$\nu\times g\in L^2(\partial\Omega,{\mathbb{R}}^3)=\bigl(L^2(\partial\Omega,{\mathbb{R}}^3)\bigr)'$.  
Thus, $\nu\times g$ acts  as a linear functional on $X$ as follows:
$$
(\nu\times g)(\phi, {\rm curl}\,\phi):=\langle \nu\times g, \nu\times\phi\rangle_{\partial\Omega}
\quad \mbox{for all }\phi\in W_T.
$$
By the Hahn-Banach theorem, there exist 
$(v_1,v_2)\in L^2(\Omega,{\mathbb{R}}^3)\times L^2(\Omega,{\mathbb{R}}^3)$ such that
$$
(\nu\times g)(\phi, {\rm curl}\,\phi)=\langle v_1,{\rm curl}\,\phi\rangle_\Omega + 
\langle v_2,\phi\rangle_\Omega \quad \mbox{for all }\phi\in W_T,
$$
where we have identified $\bigl(L^2(\Omega,{\mathbb{R}}^3)\times L^2(\Omega,{\mathbb{R}}^3)\bigr)'$
with $L^2(\Omega,{\mathbb{R}}^3)\times L^2(\Omega,{\mathbb{R}}^3)$.
We can choose $\phi\in H^1_0(\Omega,{\mathbb{R}}^3) \subset W_T$
and obtain that 
$$
0 = \ _{H^{-1}}\langle {\rm curl}\, v_1 + v_2,\phi\rangle_{H^1_0}.
$$
This gives that ${\rm curl}\, v_1 + v_2 = 0$ in $H^{-1}(\Omega,{\mathbb{R}}^3)$. 
We set $w:=-v_1\in L^2(\Omega,{\mathbb{R}}^3)$, we have 
${\rm curl}\,w=v_2\in L^2(\Omega,{\mathbb{R}}^3)$ and
\begin{equation}
\label{PWR1}
\langle \nu\times g, \nu\times\phi\rangle_{\partial\Omega}=
-\langle w,{\rm curl}\,\phi\rangle_\Omega + \langle {\rm curl}\,w,\phi\rangle_\Omega
\quad \mbox{for all }\phi\in W_T.
\end{equation}
Since $\phi\in W_T$, ${\rm Tr}_{|_{\partial\Omega}}\phi\in L^2_{\rm tan}(\partial\Omega,{\mathbb{R}}^3)$ 
it is clear\footnote{Recall that for $a,b,c\in{\mathbb{R}}^3$, the following identities hold:
$$
(a\times b)\cdot c=(b\times c)\cdot a, \quad a\times b=-b\times a,\quad
|a|^2b=(a\times b)\times a +(a\cdot b)a.
$$}
that $\phi=(\nu\times \phi)\times \nu$, so that the left-hand side of \eqref{PWR1} coincides with 
\begin{equation}
\label{PWR2}
\langle g, \phi\rangle_{\partial\Omega} \quad\mbox{ for all } \phi\in W_T,
\end{equation}
which proves \eqref{extnutimesu1}.

The existence of $C>0$ such that \eqref{estw} holds follows from the Closed Graph Theorem since 
$\{u\in L^2(\Omega,{\mathbb{R}}^3); {\rm curl}\,u\in L^2(\Omega,{\mathbb{R}}^3)\}$ is complete 
for the norm $\|u\|_2 + \|{\rm curl}\,u\|_2$. 

\vspace{.2cm}
\noindent
\ref{2}.
Assume now that $g\in L^2_{\rm tan}(\partial\Omega,{\mathbb{R}}^3)$. Let 
$w\in L^2(\Omega,{\mathbb{R}}^3)$ such that ${\rm curl}\,w\in L^2(\Omega,{\mathbb{R}}^3)$ and
\eqref{extnutimesu1} holds. Since 
$\nu\times g\in L^2(\partial\Omega,{\mathbb{R}}^3)$, we can approach it 
in $L^2(\partial\Omega,{\mathbb{R}}^3)$ by a sequence $(\varphi_n)_{n\in{\mathbb{N}}}$ 
of vector fields $\varphi_n\in H^{1/2}(\partial\Omega,{\mathbb{R}}^3)$. In particular, 
$$
\varphi_n\times \nu\longrightarrow (\nu\times g)\times\nu=g
\quad \mbox{in }L^2(\partial\Omega,{\mathbb{R}}^3)\mbox{ as }n\to\infty.
$$
By assertion \ref{1}, for each $n\in {\mathbb{N}}$ there exists $w_n\in L^2(\Omega,{\mathbb{R}}^3)$ 
such that ${\rm curl}\,w_n\in L^2(\Omega,{\mathbb{R}}^3)$ satisfying
$$
\langle \varphi_n\times\nu, \phi\rangle_{\partial\Omega}= 
\langle {\rm curl}\,w_n,\phi\rangle_\Omega -\langle w_n,{\rm curl}\,\phi\rangle_\Omega 
\quad \mbox{for all }\phi\in W_T.
$$
Thanks to the estimate \eqref{estw}, it is immediate that
$$
w_n\longrightarrow w\quad \mbox{ and }\quad {\rm curl}\,w_n\longrightarrow {\rm curl}\,w
\quad\mbox{in }L^2(\Omega,{\mathbb{R}}^3)\mbox{ as }n\to \infty.
$$
Let now $\phi\in H^1(\Omega,{\mathbb{R}}^3)$. For  $\varepsilon>0$, let 
$\phi_\varepsilon=(1+\varepsilon B_0)^{-1}\phi$ with $B_0$ as in Proposition~\ref{curlB0=B1curl}.
Then $\phi_\varepsilon\in W_T$ and thanks to \eqref{curlB}
$$
\phi_\varepsilon\longrightarrow \phi\quad \mbox{ and }\quad 
{\rm curl}\,\phi_\varepsilon=(1+\varepsilon B_1)^{-1}{\rm curl}\,\phi\longrightarrow {\rm curl}\,\phi
\quad\mbox{in }L^2(\Omega,{\mathbb{R}}^3)\mbox{ as }\varepsilon\to 0.
$$
This implies also that 
$$
\nu\times\phi_\varepsilon\longrightarrow \nu\times\phi
\quad\mbox{in }H^{-1/2}(\partial\Omega,{\mathbb{R}}^3)\mbox{ as }\varepsilon\to 0.
$$
Therefore, we have for all $\varepsilon>0$ and $n\in{\mathbb{N}}$
$$
\langle \nu\times \phi_\varepsilon,\varphi_n\rangle_{\partial\Omega}=
\langle \varphi_n\times \nu, \phi_\varepsilon\rangle_{\partial\Omega}=
 \langle {\rm curl}\,w_n,\phi_\varepsilon\rangle_\Omega 
 -\langle w_n,{\rm curl}\,\phi_\varepsilon\rangle_\Omega. 
$$
We first take the limit as $\varepsilon$ goes to $0$ and  obtain (recall that 
$\varphi_n\in H^{1/2}(\partial\Omega,{\mathbb{R}}^3)$)
$$
\ _{H^{-1/2}}\langle \nu\times \phi,\varphi_n\rangle_{H^{1/2}}
= \langle {\rm curl}\,w_n,\phi\rangle_\Omega 
 -\langle w_n,{\rm curl}\,\phi\rangle_\Omega. 
$$
Since $\phi\in H^1(\Omega,{\mathbb{R}}^3)$, the first term of the latter equation is also equal to
$\langle \varphi_n\times\nu,  \phi\rangle_{\partial\Omega}$. Taking the limit as $n$ goes 
to $\infty$ yields
$$
\langle g,\phi\rangle_{\partial\Omega} =\langle {\rm curl}\,w,\phi\rangle_\Omega 
 -\langle w,{\rm curl}\,\phi\rangle_\Omega
$$
which proves the claim made in \ref{2}.
\end{proof}

\begin{lemma}
\label{CMcI}
Let $\varphi\in H^{1/2}(\partial\Omega,{\mathbb{R}}^3)\cap 
L^2_{\rm tan}(\partial\Omega,{\mathbb{R}}^3)$. Then there exists
$v\in H^1(\Omega,{\mathbb{R}}^3)$ such that ${\rm div}\,v=0$ on $\Omega$ and
$v_{|_{\partial\Omega}}=\varphi$. 
\end{lemma}

\begin{proof} 
Let $\varphi\in H^{1/2}(\partial\Omega,{\mathbb{R}}^3)\cap 
L^2_{\rm tan}(\partial\Omega,{\mathbb{R}}^3)$. Since the trace operator 
${\rm Tr}_{|_{\partial\Omega}}:H^1(\Omega,{\mathbb{R}}^3) \to H^{1/2}(\partial\Omega,{\mathbb{R}}^3)$ 
is onto there exists $w\in H^1(\Omega,{\mathbb{R}}^3)$ 
such that $w_{|_{\partial\Omega}}=\varphi$. By \cite[Theorem~4.6]{CMcI10}, there exist
three operators $R:L^2(\Omega,{\mathbb{R}}^3)\to H^1_0(\Omega,{\mathbb{R}}^3)$, 
$S:L^2(\Omega)\to H^1_0(\Omega,{\mathbb{R}}^3)$ and 
$T:L^2(\Omega,{\mathbb{R}}^3)\to H^1_0(\Omega,{\mathbb{R}}^3)$ such that
$$
{\rm curl}\,Tu+S{\rm div}\,u=u-Ru \quad\mbox{for all }
u\in H^1(\Omega,{\mathbb{R}}^3)\mbox{ with }\nu\cdot u=0\mbox{ on }\partial\Omega
$$
(choose $n=3$, $T=T_2$, $S=T_3$ and $R=L_2$ in \cite[Theorem~4.6]{CMcI10}). We apply
this result to $u=w$ and we define
$$
v:={\rm curl}\,Tw=w-S{\rm div}\,w-Rw;
$$ 
$v$ satisfies
${\rm div}\,v=0$, $v\in H^1(\Omega,{\mathbb{R}}^3)$ and 
$v_{|_{\partial\Omega}}=w_{|_{\partial\Omega}}=\varphi$.
\end{proof}

The classical Hodge-Helmholtz decomposition asserts that the space 
$L^2(\Omega,{\mathbb{R}}^3)$ is the orthogonal direct sum 
$H\stackrel{\bot}{\oplus} G$ where
\begin{equation}
\label{def:H}
H :=\bigl\{u\in L^2(\Omega,{\mathbb{R}}^3); {\rm div}\,u=0\mbox{ in } \Omega, 
\nu\cdot u=0\mbox{ on } \partial\Omega\bigr\}
\end{equation}
and $G :=\nabla H^1(\Omega,{\mathbb{R}})$. 

\begin{remark} 
\label{densityH}
The space $H$ coincides with the closure in $L^2(\Omega,{\mathbb{R}}^3)$ of the space of
vector fields $u\in {\mathscr{C}}_c^\infty(\Omega,{\mathbb{R}}^3)$ with ${\rm div}\,u=0$ in 
$\Omega$ which we denote by ${\mathscr{D}}(\Omega)$. See, e.g., \cite[Theorem~1.4]{Temam}.
\end{remark}

We denote by $J:H\hookrightarrow L^2(\Omega;{\mathbb{R}}^3)$ the canonical embedding
and ${\mathbb{P}}:L^2(\Omega;{\mathbb{R}}^3)\to H$ the orthogonal projection. Recall that
for $u\in L^2(\Omega,{\mathbb{R}}^3)$, there exists $p\in H^1(\Omega)$ so that
${\mathbb{P}}u=u-\nabla p$. It is clear that ${\mathbb{P}} J={\rm Id}_{H}$ and that 
\begin{equation}
\label{Pcom}
\langle u,{\mathbb{P}}v\rangle_\Omega=\langle {\mathbb{P}}u, v\rangle_\Omega
\quad\mbox{for all }u,v\in L^2(\Omega;{\mathbb{R}}^3).
\end{equation}
Define now the space $V:=W_T\cap H$. Thus, for every $v\in W_T$, ${\mathbb{P}}v\in V$. 
The space $V$ will be used to define the Stokes operator 
with Robin boundary conditions in the next section. 


\section{The Robin-Stokes operator}
\label{secRS}

In this  section we define the Stokes operator with Robin boundary conditions on $\partial\Omega$. 
In order to do this we use the method of sesquilinear forms. 
We start by defining the Hodge-Laplacian with Robin boundary conditions. As in the previous 
section, $\Omega$ is a bounded domain of ${\mathbb{R}}^3$ and we suppose that it is either 
convex or has a ${\mathscr{C}}^{1,1}$-boundary. 

\medskip

Fix $\tau \in (0, \infty)$ and let $\beta:[0, \tau]\times\partial\Omega\to{\mathscr{M}}_3({\mathbb{R}})$ 
be bounded measurable on $[0, \tau] \times\partial\Omega$ such that 
\begin{align}
&0 \le\beta(t,x)\xi\cdot\xi\le M |\xi|^2
\mbox{ for almost all }(t,x)\in[0, \tau]\times\partial\Omega
\label{TES1}\\
&\quad\qquad\qquad\qquad\qquad\qquad\mbox{ and all }\xi\in{\mathbb{R}}^3
\nonumber\\[4pt]
&\beta(t,x)\mbox{ is symmetric for almost all }(t,x)\in[0, \tau]\times\partial\Omega, 
\label{TES2}\\[4pt]
&\beta(t,x)\nu(x)=\lambda(t,x)\nu(x)\mbox{ for almost all }x\in\partial\Omega, t>0, 
\label{TES3}
\end{align}
where $\lambda:[0, \tau]\times\partial\Omega\to{\mathbb{R}}$, so that a normal vector field
transformed by $\beta=\beta^\top$ remains normal at the boundary. 

Recall that $V=W_T\cap H$ and that the embedding $J$ restricted to $V$ maps $V$ to $W_T$.  
We denote this restriction  by $J_0: V \hookrightarrow W_T$. Its adjoint 
$J_0' =: {\mathbb{P}}_1:W_T'\to V'$ is then an extension of the orthogonal
projection ${\mathbb{P}}$. 

\begin{lemma}
\label{PV}
The projection ${\mathbb{P}}$ restricted to $W_T$ takes its values in $V$, so that 
${\mathbb{P}} J_0={\rm Id}_{V}$ holds.
\end{lemma}

\begin{proof}
Let $w\in W_T$. Since $W_T\subset L^2(\Omega,{\mathbb{R}}^3)$, there exists 
$\pi\in H^1(\Omega)$ such that $w=J{\mathbb{P}}w +\nabla\pi$ and  $\pi$ satisfies 
$\Delta\pi ={\rm div}\,w\in L^2(\Omega)$ and $\partial_\nu\pi=\nu\cdot w=0$ on $\partial\Omega$. 
Moreover, ${\rm curl}\,\nabla\pi=0$ in $\Omega$, so that $\nabla\pi\in W_T$. Therefore, 
${\rm div}\,J{\mathbb{P}}w=0$ in $\Omega$, 
${\rm curl}\,J{\mathbb{P}}w = {\rm curl}\,w\in L^2(\Omega,{\mathbb{R}}^3)$ and 
$\nu\cdot J{\mathbb{P}}w=0$ on $\partial\Omega$, which proves that ${\mathbb{P}}w\in V$.
\end{proof}

We are now in the situation to define the Stokes operator with Robin boundary conditions.
We consider on the Hilbert space $H$ the bilinear symmetric form 
\begin{equation}
\label{formRobinStokes}
\begin{array}{rcl}
\fra_\beta&:&V\times V \ \longrightarrow\ {\mathbb{R}}
\\[4pt]
\fra_\beta(u,v)&:=&
\langle {\rm curl}\,J_0u,{\rm curl}\,J_0v\rangle_{\Omega} +
\langle \beta\, {\rm Tr}_{|_{\partial\Omega}} J_0u, 
{\rm Tr}_{|_{\partial\Omega}} J_0v \rangle_{\partial\Omega}.
\end{array}
\end{equation}
Using the fact that $\mathbb{P}J_0 = {\rm Id}_V$ we see that the  form $\fra_\beta$ is closed. 
Therefore, there exists an operator $A_{\beta,0}: V \to V'$ associated with $\fra_\beta$ in the 
sense that 
$$
\fra_\beta(u,v) =\  _{V'}\langle A_{\beta,0}u, v \rangle_V, \quad u, v \in V.
$$
The part $A_\beta$ of $A_{\beta,0}$ on $H$, i.e.,
$$
D(A_\beta) := \{ u \in V, \exists\,v \in H: \fra_\beta(u,\phi) = 
\langle v, \phi \rangle_\Omega \ \forall\,\phi \in V \},\quad  A_\beta u := v
$$
is a self-adjoint operator on $H$. We call $A_\beta$ the Robin-Stokes operator. 

From now on, since $J$ and $J_0$ are embedding operators, we will omit to write them
to avoid too pedantic an exposition.

\begin{theorem}
\label{RSOp}
The operator $A_\beta$ is given by
\begin{align}
D(A_\beta)=&\bigl\{u\in V; {\rm curl}\,{\rm curl}\,u\in L^2(\Omega,{\mathbb{R}}^3), 
\nu\times{\rm curl}\,u=\beta u\mbox{ on }\partial\Omega\bigr\},
\label{RSOp1}\\[4pt]
A_\beta u=& {\mathbb{P}}({\rm curl}\,{\rm curl}\,u) = -\Delta u+\nabla p, 
\qquad u\in D(A_\beta), 
\nonumber
\end{align}
for some $p\in H^1(\Omega)$.

In addition, $-A_\beta$ generates an analytic semigroup of 
contractions on $H$ and $D(A_\beta^{\frac{1}{2}})=V$. 
\end{theorem}

\begin{proof}
Let $D_\beta$ be the space on the right-hand side of \eqref{RSOp1}. 
First note that, thanks to the condition \eqref{TES3} on $\beta$, 
$\beta {\rm Tr}_{|_{\partial\Omega}}u\in L^2_{\rm tan}(\partial\Omega,{\mathbb{R}}^3)$
whenever $u\in W_T$. Next, remark that for $u\in D_\beta$, 
since ${\rm curl}\,u\in L^2(\Omega,{\mathbb{R}}^3)$ and
${\rm curl}\,{\rm curl}\,u\in L^2(\Omega,{\mathbb{R}}^3)$, the integration by parts \eqref{nutimesu} 
allows to define $\nu\times{\rm curl}\,u\in H^{-1/2}(\partial\Omega,{\mathbb{R}}^3)$. Moreover,
the condition $\nu\times{\rm curl}\,u=\beta u$ on $\partial\Omega$ implies that 
$\nu\times{\rm curl}\,u\in L^2(\partial\Omega,{\mathbb{R}}^3)$ and by the obvious
fact that ${\rm div}\,{\rm curl}\,u= 0\in L^2(\Omega)$, Proposition~\ref{sobolev} yields
${\rm Tr}_{|_{\partial\Omega}}({\rm curl}\,u)\in L^2(\partial\Omega,{\mathbb{R}}^3)$.

If $u\in D_\beta$, then 
$-\Delta u={\rm curl}\,{\rm curl}\,u \in L^2(\Omega,{\mathbb{R}}^3)$
and for all $v\in V$, we have by \eqref{nutimesu}
\begin{align}
\fra_\beta(u,v)=&
\langle{\rm curl}\,u,{\rm curl}\,v\rangle_\Omega+\langle\beta u,v\rangle_{\partial\Omega}
\label{HL1}\\[4pt]
=&\langle{\rm curl}\,{\rm curl}\,u,v\rangle_\Omega
-\langle \nu \times {\rm curl}\,u, v\rangle_{\partial\Omega}
+\langle\beta u,v\rangle_{\partial\Omega}
\label{HL2}\\[4pt]
=&\langle {\mathbb{P}}({\rm curl}\,{\rm curl}\,u),v\rangle_\Omega.
\label{HL3}
\end{align}
Since ${\mathbb{P}}({\rm curl}\,{\rm curl}\,u) \in H$, 
we have then proved that for all $u\in D_\beta$, $u\in D(A_\beta)$ and 
$A_\beta u= {\mathbb{P}}({\rm curl}\,{\rm curl}\,u)$.

Conversely, let $u\in V\subset W_T$ and set $g:=\beta {\rm Tr}_{|_{\partial\Omega}}u$. As already
mentioned, $g\in L^2_{\rm tan}(\partial\Omega,{\mathbb{R}}^3)$ thanks to \eqref{TES3}.
We can then apply Lemma~\ref{extnutimesu}
to obtain $w\in L^2(\Omega,{\mathbb{R}}^3)$ with ${\rm curl}\,w\in L^2(\Omega,{\mathbb{R}}^3)$
satisfying 
\begin{equation}
\label{def:w}
\langle g, v\rangle_{\partial\Omega}=
\langle{\rm curl}\,w,v\rangle_\Omega-\langle w,{\rm curl}\,v\rangle_\Omega
\quad \mbox{ for all }v\in W_T.
\end{equation}
Therefore, for a fixed $u\in V$, we can rewrite $\fra_\beta(u,\cdot)$ as follows:
\begin{equation}
\label{fra-w}
\fra_\beta(u,v)=\langle{\rm curl}\,u,{\rm curl}\,v\rangle_\Omega
+\langle{\rm curl}\,w,v\rangle_\Omega-\langle w,{\rm curl}\,v\rangle_\Omega 
\quad \mbox{for all }v\in V.
\end{equation}
We assume now that $u\in D(A_\beta)$. Since  $A_\beta u\in H\subset L^2(\Omega,{\mathbb{R}}^3)$ 
and ${\mathbb{P}}v\in V$ for $v\in W_T$, we can write
\begin{align}
\langle A_\beta u,v\rangle_\Omega&=\langle A_\beta u,{\mathbb{P}}v\rangle_\Omega
=\fra_\beta(u,{\mathbb{P}}v) 
\label{PAbeta}\\[4pt]
&= \langle{\rm curl}\,u,{\rm curl}\,{\mathbb{P}}v\rangle_\Omega
+\langle{\rm curl}\,w,{\mathbb{P}}v\rangle_\Omega-\langle w,{\rm curl}\,{\mathbb{P}}v\rangle_\Omega
\label{PAbeta1}\\[4pt]
&=\langle{\rm curl}\,u-w,{\rm curl}\,v\rangle_\Omega
+\langle{\mathbb{P}}{\rm curl}\,w,v\rangle_\Omega.
\label{PAbeta2}
\end{align}
The last equality \eqref{PAbeta2} comes from \eqref{Pcom} and the fact that 
${\rm curl}\,{\mathbb{P}}v={\rm curl}\,v$.
Therefore we obtain
\begin{equation}
\label{PAbeta3}
\langle A_\beta u-{\mathbb{P}}{\rm curl}\,w,v\rangle_\Omega
=\langle{\rm curl}\,u-w,{\rm curl}\,v\rangle_\Omega  \quad \mbox{for all }v\in W_T.
\end{equation}
For all $v\in H^1_0(\Omega,{\mathbb{R}}^3)\subset W_T$, \eqref{PAbeta3} becomes
$$
\langle A_\beta u-{\mathbb{P}}{\rm curl}\,w,v\rangle_\Omega
=\ _{H^{-1}}\langle{\rm curl}\,({\rm curl}\,u-w),v\rangle_{H^1_0},
$$
which implies that ${\rm curl}\,({\rm curl}\,u-w)\in L^2(\Omega,{\mathbb{R}}^3)$ and 
ultimately, since ${\rm curl}\,w\in L^2(\Omega,{\mathbb{R}}^3)$, 
${\rm curl}\,{\rm curl}\,u\in L^2(\Omega,{\mathbb{R}}^3)$.

We have proved that for $u\in D(A_\beta)$, ${\rm curl}\,{\rm curl}\,u\in L^2(\Omega,{\mathbb{R}}^3)$. 
It remains to identify $A_\beta u$ and the boundary condition $\nu\times{\rm curl}\,u=\beta u$ 
on $\partial\Omega$ for $u\in D(A_\beta)$. Note that this condition is well defined  thanks to 
\eqref{nutimesu} since 
${\rm curl}\,u\in L^2(\Omega,{\mathbb{R}}^3)$ ($u\in D(A_\beta)\subset V\subset W_T$) and
${\rm curl}\,{\rm curl}\,u\in L^2(\Omega,{\mathbb{R}}^3)$. By definition \eqref{formRobinStokes}
of $\fra_\beta$ and thanks to \eqref{Pcom}, we have for all $v\in {\mathscr{D}}(\Omega)$ (recall that
${\mathscr{D}}(\Omega)=\bigl\{w\in {\mathscr{C}}_c^\infty(\Omega,{\mathbb{R}}^3), {\rm div}\,w=0
\mbox{ in }\Omega\bigr\}$ has been defined in Remark~\ref{densityH})
\begin{align}
\langle A_\beta u,v\rangle_\Omega
&=\fra_\beta(u,v) 
= \langle{\rm curl}\,u,{\rm curl}\,v\rangle_\Omega
\nonumber\\[4pt]
&=\langle{\rm curl}\,{\rm curl}\,u,v\rangle_\Omega 
= \langle{\rm curl}\,{\rm curl}\,u,{\mathbb{P}}v\rangle_\Omega
\nonumber\\[4pt]
&=\langle{\mathbb{P}}({\rm curl}\,{\rm curl}\,u),v\rangle_\Omega ,
\end{align}
since ${\mathbb{P}}v=v$. This proves that 
$A_\beta u={\mathbb{P}}({\rm curl}\,{\rm curl}\,u)$ since ${\mathscr{D}}(\Omega)$ is dense in $H$ 
(see Remark~\ref{densityH}).

Now, let $v\in V$ and  recall that 
${\rm Tr}_{|_{\partial\Omega}}v\in H^{1/2}(\partial\Omega,{\mathbb{R}}^3)$.
We have then by \eqref{nutimesu}
\begin{align*}
\langle {\mathbb{P}}({\rm curl}\,{\rm curl}\,u),v\rangle_\Omega
&= \langle A_\beta u,v\rangle_\Omega\ =\ \fra_\beta(u,v)
\nonumber\\[4pt]
&=\langle{\rm curl}\,u,{\rm curl}\,v\rangle_\Omega
+\langle\beta u,v\rangle_{\partial\Omega}
\\[4pt]
&=\langle{\rm curl}\,{\rm curl}\,u,v\rangle_\Omega
-\ _{H^{-1/2}}\langle\nu\times{\rm curl}\,u,v\rangle_{H^{1/2}}
+\langle\beta u,v\rangle_{\partial\Omega}
\nonumber\\[4pt]
&=\langle{\mathbb{P}}({\rm curl}\,{\rm curl}\,u),v\rangle_\Omega
-\ _{H^{-1/2}}\langle\nu\times{\rm curl}\,u,v\rangle_{H^{1/2}}
+\langle\beta u,v\rangle_{\partial\Omega},
\end{align*}
which proves that 
\begin{equation}
\label{PAbeta4}
\ _{H^{-1/2}}\langle\beta u-\nu\times{\rm curl}\,u,v\rangle_{H^{1/2}}=0
\quad \mbox{for all }v\in V. 
\end{equation}
Let $\varphi\in H^{1/2}_{\rm tan}(\partial\Omega,{\mathbb{R}}^3)$ be arbitrary. By Lemma~\ref{CMcI},
we can find $v\in V$ such that $v_{|_{\partial\Omega}}=\varphi$ on $\partial\Omega$. Therefore,
\eqref{PAbeta4} implies that for all $\varphi\in H^{1/2}_{\rm tan}(\partial\Omega,{\mathbb{R}}^3)$
\begin{equation}
\label{PAbeta5}
\ _{H^{-1/2}}\langle\beta u-\nu\times{\rm curl}\,u,\varphi\rangle_{H^{1/2}}=0,
\end{equation}
With $w\in L^2(\Omega,{\mathbb{R}}^3)$ such that ${\rm curl}\,w\in L^2(\Omega,{\mathbb{R}}^3)$
satisfying $\beta u=\nu\times w$ in $H^{-1/2}(\partial\Omega,{\mathbb{R}}^3)$ as in 
Lemma~\ref{extnutimesu}, it follows from \eqref{PAbeta5} that $w_1:=w-{\rm curl}\,u$ satisfies
\begin{equation}
\label{eqWT}
\langle {\rm curl}\,w_1,v\rangle_\Omega-\langle w_1,{\rm curl}\,v\rangle_\Omega=0
\quad\mbox{for all }v\in W_T.
\end{equation}
Let now $v\in H^1(\Omega,{\mathbb{R}}^3)$ and denote for $\varepsilon>0$, 
$v_\varepsilon=(1+\varepsilon B_0)^{-1}v$ (recall that the operator $B_0$ has been
defined in \eqref{def:B0}). It is clear that $v_\varepsilon\in W_T$ for all $\varepsilon>0$
and 
$$
v_\varepsilon\longrightarrow v\quad\mbox{ in }L^2(\Omega,{\mathbb{R}}^3)\quad\mbox{ as }
\varepsilon\to 0.
$$
Moreover, thanks to \eqref{curlB}, we have
$$
{\rm curl}\,v_\varepsilon=(1+\varepsilon B_1)^{-1}{\rm curl}\,v
\longrightarrow v\quad\mbox{ in }L^2(\Omega,{\mathbb{R}}^3)\quad\mbox{ as }\varepsilon\to 0.
$$
Applying \eqref{eqWT} to $v_\varepsilon$ and taking the limit as $\varepsilon\to 0$, we obtain
$$
0=\langle {\rm curl}\,w_1,v_\varepsilon\rangle_\Omega-
\langle w_1,{\rm curl}\,v_\varepsilon\rangle_\Omega
\longrightarrow \langle {\rm curl}\,w_1,v\rangle_\Omega-\langle w_1,{\rm curl}\,v\rangle_\Omega
\quad\mbox{ as }\varepsilon\to 0.
$$
It follows then that $\nu\times w_1=0$ in $H^{-1/2}(\partial\Omega,{\mathbb{R}}^3)$ and therefore
$\beta u-\nu\times{\rm curl}\,u=0$ in $H^{-1/2}(\partial\Omega,{\mathbb{R}}^3)$.

Finally, the fact that $-A_\beta$ generates an analytic semigroup of contractions  follows from the 
fact that $A_\beta$ is a non-negative self-adjoint operator. 
The equality $D(A_\beta^{\frac{1}{2}})=V$ is a standard result for symmetric bilinear closed forms
(see \cite{Li62} and \cite{Ka62}).
\end{proof}

\begin{corollary}
\label{corcurl}
If $u \in D(A_\beta)$ then ${\rm curl}\, u \in L^3(\Omega, {\mathbb R}^3)$ and 
there exists a constant $C_\Omega$ independent of $u$ such that 
$$
\| {\rm curl}\, u \|_3 \le C_\Omega \left( \| A_\beta u \|_H + (\|\beta\|_\infty+1) \| u \|_V \right).
$$
\end{corollary}

\begin{proof} 
Let $u \in D(A_\beta)$. By Theorem~\ref{RSOp}, ${\rm curl}\,u\in L^2(\Omega,{\mathbb{R}}^3)$,
${\rm curl}\,{\rm curl}\,u\in L^2(\Omega,{\mathbb{R}}^3)$ and 
$\nu\times{\rm curl}\,u=\beta u\in L^2(\partial\Omega,{\mathbb{R}}^3)$. Therefore, by
Proposition~\ref{sobolev}, ${\rm curl}\,u\in H^{1/2}(\Omega,{\mathbb{R}}^3)$ with
the estimate
\begin{align*}
\|{\rm curl}\,u\|_{H^{1/2}(\Omega,{\mathbb{R}}^3)}
&\le C\bigl(\|{\rm curl}\,u\|_{L^2(\Omega,{\mathbb{R}}^3)}+
\|{\rm curl}\,{\rm curl}\,u\|_{L^2(\Omega,{\mathbb{R}}^3)}
+\|\beta u\|_{L^2(\partial\Omega,{\mathbb{R}}^3)}\bigl)
\\[4pt]
&\le C\bigl((\|\beta\|_\infty + 1)\|u\|_{V}+
\|{\rm curl}\,{\rm curl}\,u\|_{L^2(\Omega,{\mathbb{R}}^3)}\bigl).
\end{align*}
This latter estimate together with the following Sobolev embedding valid in dimension 3
$$
H^{1/2}(\Omega,{\mathbb{R}}^3)\hookrightarrow
L^3(\Omega,{\mathbb{R}}^3)
$$
proves the corollary. 
\end{proof}

\section{Maximal regularity for non-autonomous equations}
\label{sec:main}

Our aim in this section is to show maximal regularity for the Stokes problem.  
We first recall some recent results on maximal regularity for evolution equations associated 
with time-dependent sesquilinear forms.

Let ${\mathcal{H}}$ be a Hilbert space and let ${\mathcal{V}}$ be another
Hilbert space with dense and continuous embedding in ${\mathcal{H}}$. Consider a family 
of sesquilinear forms $(\fra(t))_{0 \le t \le \tau}$
such that $D(\fra(t))={\mathcal{V}}$ for all $t$. We suppose that  
$(\fra(t))_{0 \le t \le \tau}$ is uniformly bounded in the sense that  there exists a 
constant $M$ independent of $t$ such that 
\begin{equation}
\label{continuity}
|\fra(t,u,v)|\le M\|u\|_{\mathcal{V}}\|v\|_{\mathcal{V}} 
\end{equation}
for all $u, v\in{\mathcal{V}}$. Here $\|v\|_{\mathcal{V}}$ denotes the norm of ${\mathcal{V}}$. 
We also suppose that $(\fra(t))_{0 \le t \le \tau}$ is quasi-coercive, i.e., 
there exists $\delta>0$ and $\mu \in{\mathbb{R}}$ such that
\begin{equation}
\label{coercive}
\delta \|u\|_{\mathcal{V}}^2 \le\fra(t,u,u)+\mu \|u\|_{\mathcal{H}}^2,
\end{equation} 
for all $u \in {\mathcal{V}}$.

\noindent
For each fixed $t$, the form $\fra(t)$ is closed. Denote by 
${\mathcal{A}}(t) : {\mathcal{V}} \to{\mathcal{V}}'$ the operator associated with 
$\fra(t)$ in the sense that 
$$ 
\fra(t,u,v) = \ _{{\mathcal V}'}\langle{\mathcal{A}}(t)u,v\rangle_{\mathcal{V}}, 
\, \forall\ u, v \in {\mathcal{V}}.
$$
The operator associated with $\fra(t)$ on ${\mathcal{H}}$ is the part of ${\mathcal{A}}(t)$. 
That is, 
$$
D(A(t)) = \{u\in{\mathcal{V}}, {\mathcal{A}}(t)u\in{\mathcal{H}} \}, 
\quad {\mathcal{A}}(t)u=A(t) u.
$$
We consider now the evolution problem
\begin{equation}  
\label{eq:evol-eq} \tag{P}
\left\{
\begin{array}{rcl}
\partial_t u(t) + A(t)u(t)&=&f(t) 
\\[4pt]
u(0)&=&u_0. 
\end{array}\right.
\end{equation}
One says that (P) has $L^p$ maximal regularity in ${\mathcal{H}}$ if for every 
$f \in L^p(0,\tau,{\mathcal{H}})$ there exists a unique $u \in W^{1,p}(0,\tau,{\mathcal{H}})$
which satisfies the problem in the $L^p$-sense. Note that one has in addition that 
$t \mapsto A(t)u(t) $ is in $L^p(0,\tau,{\mathcal{H}})$. 

\noindent
Maximal regularity for non-autonomous equations in ${\mathcal{H}}$ has been investigated 
recently in the context of operators associated with forms as we described above. The following 
is a particular case of a result proved in \cite{HO}.  

\begin{theorem}
\label{thm:HO}
Let $(\fra(t))_{0 \le t \le \tau}$  be a family of sesquilinear forms satisfying the previous  
conditions \eqref{continuity} and \eqref{coercive}. Suppose in addition that 
$t \mapsto \fra(t)$ is piecewise $\alpha-$H\"older
continuous for some $\alpha>1/2$ in the sense that there exist 
$t_0 = 0 < t_1< ...< t_k = \tau$ and constants $M_i$ such that the restriction of 
$t\mapsto \fra(t,.,.)$ to $(t_i,t_{i+1})$ satisfies
\begin{equation}
\label{H12}
|\fra(t,u,v)-\fra(s,u,v)|\le M_i | t - s |^{\alpha} \|u\|_{\mathcal{V}} \|v\|_{\mathcal{V}}  \quad
\mbox{for all }u,v\in{\mathcal{V}}.
\end{equation}
Then the Cauchy problem {\rm(P)} has $L^2$-maximal regularity for all 
$u_0 \in D({(w_0 + A(0))}^{1/2})$.
\end{theorem}

Note that if the form $\fra(0)$ is symmetric  then $D({(\mu + A(0))}^{1/2})={\mathcal{V}}$. 
Recall also that if the $L^2$-maximal regularity holds for (P) then the solution $u$
satisfies the a priori estimate
\begin{equation}
\label{apriori-beta}
\|u\|_{H^1(0,\tau,{\mathcal{H}})}+\|A(t)u(t) \|_{L^2(0, \tau,{\mathcal{H}})} 
\le C \bigl( \|f\|_{L^2(0, \tau,{\mathcal{H}})}+\|u_0\|_{\mathcal{V}}\bigr).
\end{equation}
Now we turn back to the Robin-Stokes operator $A_\beta$. As previously,  
$\Omega$ denotes  a bounded domain of ${\mathbb{R}}^3$ which is either 
${\mathscr{C}}^{1,1}$ or convex. Let ${\mathcal{H}}:= H$ defined by  \eqref{def:H}, that is 
$$
H:=\bigl\{u\in L^2(\Omega,{\mathbb{R}}^3); {\rm div}\,u=0\mbox{ in } \Omega, 
\nu\cdot u=0\mbox{ on } \partial\Omega\bigr\}
$$
and $\fra_\beta$ the form defined by \eqref{formRobinStokes}. We assume in addition to \eqref{TES1}, 
\eqref{TES2} and \eqref{TES3} that 
$t \mapsto \beta(t,x)$ is piecewise H\"older continuous of order $\alpha>1/2$. 
This means that there exist $t_i$, $0\le i\le n$ such that  
$[0, \tau] = \cup_{i=0}^n [t_i, t_{i+1}]$ and constants $M_i$ such that on each interval 
$(t_i, t_{i+1})$, $\beta$ is the restriction of some $\widetilde{\beta}$ such that 
\begin{equation}
\label{Holder-beta}
\|\widetilde{\beta}(t,x)-\widetilde{\beta}(s,x) \|_{{\mathscr{M}}_3} 
\le M_i | t -s |^\alpha \quad \mbox{for all } t,s \in [t_i, t_{i+1}]  \mbox{ and a.e. }x \in\partial\Omega.
\end{equation}
Here $\|\cdot\|_{{\mathscr{M}}_3}$ denotes the operator norm in ${\mathscr{M}}_3$. 

The family of forms  $\fra_\beta = \fra_{\beta(t, \cdot)}$, $0 \le t \le \tau$, satisfies the assumptions 
of Theorem~\ref{thm:HO}. In order to check \eqref{H12} we write for $u,v\in V$ and 
$t, s \in (t_i, t_{i+1})$ 
\begin{align*}
|\fra_{\beta(t,\cdot)}(u,v)-\fra_{\beta(s,\cdot)}(u,v)| &= 
\langle (\beta(t,\cdot) - \beta(s,\cdot) )u, v\rangle_{\partial \Omega}
\\[4pt]
&\le \sup_{x \in \partial\Omega} \|\beta(t,x)-\beta(s,x)\|_{{\mathscr{M}}_3} 
\|{\rm Tr}_{|{\partial\Omega}}u\|_{L^2(\partial\Omega,{\mathbb{R}}^3)} 
\|{\rm Tr}_{|{\partial\Omega}}v \|_{L^2(\partial\Omega,{\mathbb{R}}^3)}
\\[4pt]
&\le C M_i  | t -s |^\alpha  \|u\|_V \|v\|_V.
\end{align*}
The last inequality follows from \eqref{Holder-beta} and Proposition~\ref{sobolev}.  
Therefore we conclude that $L^2$-maximal regularity holds for the Robin-Stokes operator 
$A_\beta$ on the Hilbert space $H$.

\begin{theorem}
\label{thm:reg}
Under the above assumptions, for every $u_0 \in V$ and every $f \in L^2(0,\tau,H)$ there exists 
a unique $u \in H^1(0,\tau,H)$ such that $u(t) \in D(A_{\beta(t, \cdot)})$ for almost all $t \in [0,\tau]$ 
and 
\begin{equation}  
\label{PR-A-BETA} 
\left\{
\begin{array}{rcl}
\partial_t u(t, \cdot) + A_{\beta(t,\cdot )} u(t, \cdot) &=& f(t) 
\\[4pt]
u(0)&=&u_0. 
\end{array}\right.
\end{equation}
In addition there exists a constant $C_{MR}$ independent of $t$, $f$ and $u_0$ such that 
\begin{equation}
\label{apriori-beta1}
\|u\|_{H^1(0,\tau,H)} + \|A_{\beta(t, \cdot)}u(t)\|_{L^2(0,\tau,H)} 
\le C_{MR} \bigl(\|f\|_{L^2(0,\tau,H)} + \|u_0\|_V \bigr).
\end{equation}
\end{theorem}

Note that if \eqref{Holder-beta} holds with $\alpha=1$ then we can apply the results from \cite{ADLO} 
and obtain the previous theorem with the additional information that  the solution 
$u \in {\mathscr{C}}([0,\tau],V)$. In particular, $u\in L^\infty(0,\tau,V)$. This latter property is 
not covered by the results in \cite{HO} when \eqref{Holder-beta} holds for some $\alpha>1/2$. 
As we will need this in the next section we prove it here. We do this in a general setting. 

As in the beginning of this section, let $(\fra(t))_{0 \le t \le \tau}$ be a family of symmetric  
forms on a Hilbert space ${\mathcal{H}}$ which satisfy \eqref{continuity} and \eqref{coercive}. 
Suppose that $t \mapsto \fra(t)$ is piecewise $\alpha-$H\"older continuous for some 
$\alpha>1/2$ (see Theorem~\ref{thm:HO}). We define the space of maximal regularity 
\begin{equation}
\label{def:E}
E := \bigl\{u\in H^1(0,\tau,{\mathcal{H}}), u(t) \in D(A_{\beta(t)})\, {\rm a.e.}, 
t\mapsto A(t)u(t) \in L^2(0, \tau,{\mathcal{H}}) 
\mbox{ and } u(0)\in{\mathcal{V}}\bigr\}.
\end{equation}
The space $E$ is  endowed with the natural norm
$$
\|u\|_E:= \|u(\cdot)\|_{ H^1(0,\tau,{\mathcal{H}})} + \|A(\cdot)u(\cdot)\|_{L^2(0,\tau;{\mathcal{H}})} 
+ \|u(0)\|_{\mathcal{V}}.
$$
Clearly, $(E, \|\cdot\|_E)$ is a Banach space. Note that if $u(\cdot)\in H^1(0,\tau,{\mathcal{H}})$ 
then $u \in {\mathscr{C}}([0,\tau],{\mathcal{H}})$ and hence $u(0)$, needed in the definition of 
$E$, is well defined.

\begin{proposition}
\label{pro-bdd}
The space $E$ is continuously embedded into $L^\infty(0,\tau,{\mathcal{V}})$.
\end{proposition}

\begin{proof} 
First by adding a positive constant to $A(t)$, it is clear that we may suppose without loss 
of generality that \eqref{coercive} holds with $\mu = 0$.

\noindent
Let $u\in E$ and set  $f:=\partial_t u+A(\cdot)u(\cdot) \in L^2(0, \tau,{\mathcal{H}})$. 
As in \cite{HO}, taking the derivative of $s\mapsto v(s):=e^{-(t-s){\mathcal{A}}(t)}u(s)$ for 
$0<s\le t<\tau$ and then integrating from $0$ to $t$ it follows  that 
\begin{equation}
\label{uu}
u(t) = \int_0^t e^{-(t-s){\mathcal{A}}(t)} ({\mathcal{A}}(t)-{\mathcal{A}}(s))u(s)\,{\rm d}s 
+ e^{-tA(t)}u(0) + \int_0^t e^{-(t-s) A(t)} f(s)\,{\rm d}s.
\end{equation}
We estimate the norm in ${\mathcal{V}}$ of each term. Recall that $-{\mathcal{A}}(t)$ 
generates a bounded holomorphic semigroup in ${\mathcal{V}}'$ (see \cite[Chapter~1]{Ou05}) 
with bound independent of $t \in [0, \tau]$ thanks to \eqref{continuity} and \eqref{coercive}. 
In particular, there exist a constant $C$ such that for all $s > 0$ and $t \in [0, \tau]$
\begin{equation}
\label{hol}
\|e^{-s{\mathcal{A}}(t)}\|_{{\mathscr{L}}({\mathcal{V}}',{\mathcal{V}})} 
\le \delta^{-1} \|{\mathcal{A}}(t) e^{-s{\mathcal{A}}(t)}\|_{{\mathscr{L}}({\mathcal{V}}')} 
\le \frac{C}{s}.
\end{equation}
Therefore,
\begin{align*}
\Bigl\|\int_0^t e^{-(t-s){\mathcal{A}}(t)} ({\mathcal{A}}(t)-{\mathcal{A}}(s))u(s)\,{\rm d}s\Bigr\|_{{\mathcal{V}}} 
&\le \int_0^t \frac{C}{t-s}\,\|({\mathcal{A}}(t)-{\mathcal{A}}(s))u(s) \|_{{\mathcal{V}}'} \,{\rm d}s
\\[4pt]
&\le \int_0^t \frac{C \omega(t-s)}{t-s}\,\|u(s)\|_{\mathcal{V}}\,{\rm d}s
\end{align*} 
where $r\mapsto\omega(r)$ is piecewise $\alpha-$H\"older continuous on $[0, \tau]$ with $\alpha>1/2$ 
by assumptions. By the Cauchy-Schwarz inequality we conclude that
\begin{equation}
\label{a1}
\Bigl\|\int_0^t e^{-(t-s){\mathcal{A}}(t)}({\mathcal{A}}(t)-{\mathcal{A}}(s))u(s)\,{\rm d}s\Bigr\|_{{\mathcal{V}}} 
\le C' \Bigl(\int_0^t \|u(s)\|_{\mathcal{V}}^2\,ds \Bigr)^{1/2}.
\end{equation}
The second term is easily estimated since the semigroup $(e^{-sA(t)})_{s \ge 0}$ is uniformly 
bounded on ${\mathcal{V}}$ (see again \cite[Chapter~1]{Ou05}). Thus
\begin{equation}
\label{a2}
\|e^{-tA(t)}u(0) \|_{\mathcal{V}} \le C \|u(0)\|_{\mathcal{V}} \quad\mbox{for all }t\ge0. 
\end{equation}
It remains to estimate the third term. Set $v(s) := \int_0^s e^{-(s-r)A(t)} f(r)\,{\rm d}r$, $s\ge 0$. 
The function $v$ satisfies
$$ 
\partial_s v + A(t) v = f, \quad v(0) = 0.
$$
Fix $\varepsilon > 0$. Since $A(t)^{1/2} e^{-\varepsilon A(t)}$ is a bounded operator on 
$\mathcal{H}$ we have
\begin{align*}
\frac{1}{2}\, \frac{d}{ds} \|A(t)^{1/2}e^{-\varepsilon A(t)} v(s)\|_{\mathcal{H}}^2 
&= ( A(t)^{1/2} e^{-\varepsilon A(t)} v'(s), A(t)^{1/2} e^{-\varepsilon A(t)}v(s))
\\[4pt]
&= ( -A(t) v(s) + f(s) , A(t)e^{-2\varepsilon A(t)} v(s)).
\end{align*}
Thus,
\begin{align*}
\frac{1}{2}\, \frac{d}{ds} \|A(t)^{1/2}e^{-\varepsilon A(t)} v(s)\|_{\mathcal{H}}^2  + 
\|A(t) e^{-\varepsilon A(t)} v(s)\|_{\mathcal{H}}^2
&= (f(s) , A(t) e^{-2\varepsilon A(t)} v(s))\\[4pt]
&\le \frac{1}{2} \, \|f(s)\|_{\mathcal{H}}^2+\frac{1}{2}\,\|A(t)e^{-2\varepsilon A(t)}v(s)\|_{\mathcal{H}}^2.
\end{align*}
Next we integrate from $0$ to $t$ and then letting $\varepsilon\to 0$ it follows that 
$$
\Bigl\|A(t)^{1/2} \int_0^t e^{-(t-r) A(t)} f(r)\,{\rm d}r\Bigr\|_{\mathcal{V}}^2
\le \|f\|_{L^2(0,\tau, {\mathcal{H}})}^2.
$$
From the  coercivity assumption \eqref{coercive} with $\mu= 0$, it follows that 
\begin{equation}
\label{a3}
\Bigl\| \int_0^t e^{-(t-r)A(t)} f(r)\,{\rm d}r\Bigr\|_{\mathcal{V}}^2 
\le \delta^{-1} \|f\|_{L^2(0,\tau,{\mathcal{H}})}^2.
\end{equation}
We obtain from \eqref{uu} and the forgoing estimates \eqref{a1}--\eqref{a3} that for some constant 
$C_0>0$
$$ 
\|u(t)\|_{\mathcal{V}}^2 \le C_0 \Bigl[\int_0^t \|u(s)\|_{\mathcal{V}}^2\,{\rm d}s 
+ \|u(0)\|_{\mathcal{V}}^2+\|f\|_{L^2(0,\tau,{\mathcal{H}})}^2 \Bigr].
$$
It follows from Gronwall's lemma that 
$$
\|u(t) \|_{\mathcal{V}}^2 \le C_0\,e^{C_0 \tau} \Bigr[\|u(0)\|_{\mathcal{V}}^2 
+\|f\|_{L^2(0, \tau,{\mathcal{H}})}^2 \Bigr].
$$
Replacing $f(t)$ by its expression  $f(t)=\partial_t u(t)+A(t)u(t)$, the conclusion of the proposition 
follows. 
\end{proof}


\section{The Navier-Stokes system with Robin boundary conditions}
\label{secNS}

As in the previous sections, $\Omega$ denotes a bounded ${\mathscr{C}}^{1,1}$ or convex domain 
of $\mathbb{R}^3$ and $\beta:[0, \tau]\times\partial\Omega\to{\mathscr{M}}_3({\mathbb{R}})$ 
satisfies \eqref{TES1})--\eqref{TES3} and 
\eqref{Holder-beta} for some $\alpha>\frac{1}{2}$. Recall from Section~\ref{secRS} that 
$$
H=\bigl\{u\in L^2(\Omega,{\mathbb{R}}^3); {\rm div}\,u=0\mbox{ in } \Omega, 
\nu\cdot u=0\mbox{ on } \partial\Omega\bigr\}
$$
and 
$$
V=\bigl\{u\in L^2(\Omega,{\mathbb{R}}^3); {\rm div}\,u = 0 \mbox{ in } \Omega,  
{\rm curl}\,u\in L^2(\Omega,{\mathbb{R}}^3) \mbox{ and } \nu\cdot u=0 \mbox{ on }\partial\Omega\bigr\}.
$$
The latter space is the domain of the bilinear symmetric form which gives rise to the 
Robin-Stokes operator $A_\beta$ defined in Section~\ref{secRS}.

\noindent
We consider the Navier-Stokes system with Robin-type boundary conditions on the time interval 
$[0,\tau]$
\begin{equation}  
\label{PR-A-BETA1} \tag{NS}
\left\{
\begin{array}{rclcl}
\partial_t u-\Delta u + \nabla \pi-u \times {\rm curl}\,u &=& 0 &\mbox{in}&[0,\tau]\times\Omega
\\[4pt]
{\rm div}\,u &=& 0 &\mbox{in}&[0,\tau]\times\Omega 
\\[4pt]
\nu \cdot u &=& 0 &\mbox{on} &[0,\tau]\times\partial\Omega 
\\[4pt]
\nu \times {\rm curl}\,u &=& \beta u &\mbox{on} &[0,\tau]\times\partial\Omega
\\[4pt]
u(0)&=&u_0&\mbox{in}&\Omega.
\end{array}\right.
\end{equation}
Our main result in this section is the following existence and uniqueness result for (NS).

\begin{theorem}
\label{thm:NS}
There exists $\epsilon > 0$ such that for every $u_0 \in V$ with $\|u_0\|_V\le\epsilon$, 
there exists a unique $u\in H^1(0,\tau,H)$ with $t \mapsto A_{\beta(t)}u(t) \in L^2(0,\tau, H)$ and  
$\pi \in L^2(0,\tau, H^1(\Omega))$ such that
$(u, \pi)$ satisfies (NS) for a.e. $(t,x) \in [0, \tau] \times \Omega$. In addition there exists 
a constant $C$ independent of $u$ and $\pi$ such that 
\begin{equation}
\label{estNS}
\|u\|_{H^1(0,\tau,H)} + \|-\Delta u\|_{L^2(0,\tau,L^2(\Omega,{\mathbb{R}}^3))}
+\|\nabla\pi\|_{L^2(0,\tau,L^2(\Omega,{\mathbb{R}}^3))} \le C \epsilon.
\end{equation}
\end{theorem}

\begin{proof}  
Recall the maximal regularity space 
$$
E=\bigl\{u\in H^1(0,\tau,H); u(t) \in D(A_{\beta(t)})\, {\rm a.e.}, t\mapsto A_{\beta(t)}u(t)\in L^2(0,\tau,H)\mbox{ and }u(0)\in V\bigr\}.
$$ 
For all $u\in E$, $u(t)\in D(A_{\beta(t)})$ for a.e. $t \in [0, \tau]$. Then by Corollary~\ref{corcurl}
$$
\|{\rm curl}\,u(t)\|_3\le C_\Omega\|A_{\beta(t)}u(t)\|_H + C( \|\beta\|_\infty+1) \|u(t)\|_V.
$$
Using Proposition~\ref{pro-bdd} and taking the $L^2$-norm in time it follows that 
\begin{equation}
\label{a10}
\|{\rm curl}\,u\|_{L^2(0, \tau,L^3(\Omega,{\mathbb{R}}^3))}  
\le C_\Omega \|u\|_E + C (\|\beta\|_\infty+1)\|u\|_E = C_1\|u\|_E.
\end{equation}
On the other hand, by \eqref{WdansH1}, the classical Sobolev embedding of  
$H^1(\Omega) $ into $L^6(\Omega)$ in dimension 3 and Proposition~\ref{pro-bdd}, 
there exists a constant $C_2$ such that for every $u\in E$ 
\begin{equation}
\label{a11}
\|u\|_{L^\infty(0, \tau,L^6(\Omega,{\mathbb{R}}^3))} \le C_2\|u\|_E.
\end{equation}
Let $u_0 \in V$. By Theorem~\ref{thm:reg}, there exists a solution $a\in E$ 
of the problem
\begin{equation}  
\label{NSa}
\left\{
\begin{array}{rcl}
\partial_t a + A_{\beta(t)}a &=& 0
\\[4pt]
a(0)&=&u_0.
\end{array}\right.
\end{equation}
with 
\begin{equation}
\label{trucmuche}
\|a\|_E \le C_{MR} \|u_0\|_V.
\end{equation}
Let $u, v \in E$ and set $f :=\frac{1}{2}\,\mathbb{P}(u\times{\rm curl}\,v+v\times{\rm curl}\,u)$. 
By \eqref{a10} and \eqref{a11}, $f\in L^2(0,\tau, H)$ and 
\begin{equation}
\label{a12}
\|f\|_{L^2(0, \tau,H) }\le C_1C_2 \|u\|_E\|v\|_E.
\end{equation}
Again by Theorem~\ref{thm:reg} there exists $w$ solution of 
\begin{equation}  
\label{NSw}
\left\{
\begin{array}{rcl}
\partial_t w + A_{\beta(t)} w &=&f 
\\[4pt]
w(0)&=&0.
\end{array}\right.
\end{equation}
In addition, $w \in E$ and satisfies $\|w\|_E \le C_{MR} \|f\|_{L^2(0,\tau,H)}$. 

We define the bilinear application
$$ 
B: E \times E \to E, \quad (u,v) \mapsto w.
$$
Then the latter estimate gives
\begin{equation}
\label{NSuvw}
\|B(u,v)\|_E = \|w\|_E \le C_{MR} \|f\|_{L^2(0,\tau,H)}.
\end{equation}
Thus we have from \eqref{a12}
\begin{equation}
\label{uvw}
\|B(u,v)\|_E \le C_{MR} C_1 C_2 \|u\|_E \|v\|_E.
\end{equation}
We now use Picard's contraction principle. Let $\delta>0$ such that 
$\delta < \frac{1}{4C_{MR}C_1 C_2}$. If  $\|a\|_E \le \delta$, we define  
$$
\begin{array}{rcrcl}
T &:& \overline{B}_E(0, 2 \delta) &\to&  \overline{B}_E(0, 2 \delta) 
\\[4pt]
&& v & \mapsto & a + B(v,v).
\end{array}
$$
To see that $T$ maps $\overline{B}_E(0, 2 \delta)$ into itself, we use \eqref{uvw} so that for 
$v \in \overline{B}_E(0, 2 \delta)$
\begin{align*}
\|T(v)\|_E &\le \|a\|_E + \|B(v,v)\|_E
\\[4pt]
&\le \delta + C_{MR} C_1 C_2 \|v\|_E^2
\\[4pt]
&\le \delta + 4C_{MR}C_1 C_2 \delta^2 \ \le \ 2\delta.
\end{align*}
Moreover, the map $T$ is a strict contraction. Indeed, for every $v,w\in \overline{B}_E(0, 2 \delta)$
\begin{align*}
\|T(v)-T(w) \|_E &=  \| B(v+w,v-w)\|_E
\\[4pt]
&\le C_{MR} C_1 C_2 \|v+w\|_E \|v-w\|_E
\\[4pt]
&\le 4\delta C_{MR} C_1 C_2 \|v-w\|_E.
\end{align*}
Therefore there exists a unique $u\in\overline{B}_E(0, 2 \delta)$ satisfying $u = a + B(u,u)$. 
By \eqref{trucmuche}, the condition $\|a\|_E \le\delta$ is satisfied if 
$\|u_0\|_V \le \epsilon := \frac{\delta}{C_{MR}}$. 
It remains to prove that $u$ is a solution of (NS) for a.e. $(t,x) \in [0, \tau]\times\Omega$. 

Since $u = a + B(u,u)$  with $a$ the solution of \eqref{NSa} and $w = B(u,u)$ the solution of 
\eqref{NSw} with $v= u$ we obtain  
\begin{align*}
\partial_t u &= \partial_t a + \partial_t B(u,u)
\\[4pt]
&= -A_\beta a - A_\beta B(u,u) + {\mathbb{P}}(u\times{\rm curl}\,u)
\\[4pt]
&= -A_\beta u + {\mathbb{P}}(u\times{\rm curl}\,u).
\end{align*}
Since $u\in E$, $t\mapsto A_{\beta(t)}u(t)\in L^2(0,\tau, H)$ and hence
by Theorem~\ref{RSOp}, 
$$
t\mapsto {\rm curl}\,{\rm curl}\,u(t)=-\Delta u(t) \in L^2(0,\tau, L^2(\Omega,{\mathbb{R}}^3).
$$
Thus, $A_\beta u = -\Delta u + \nabla q$ with $q \in L^2(0,\tau,H^1(\Omega))$. 
In addition
$$ 
\nu\cdot u = 0 \quad\mbox{and} \quad \nu\times{\rm curl}\,u=\beta u
$$
for a.e. $(t,x) \in (0, \tau)\times\partial\Omega$. By the definition of ${\mathbb{P}}$
and integrability properties \eqref{a11} (for $u$) and \eqref{a10} (for ${\rm curl}\,u$),
${\mathbb{P}}(u\times{\rm curl}\,u) = u\times{\rm curl}\,u + \nabla p$ with 
$p\in L^2(0,\tau, H^1(\Omega))$.
Therefore, if we take $\pi := p + q$ we see that $(u, \pi)$ satisfy (NS) for a.e. 
$(t,x)\in[0,\tau]\times\Omega$. 
\end{proof}

{\small
\bibliographystyle{amsplain}

}
\end{document}